\documentclass{amsart}
\usepackage[utf8]{inputenc}
\usepackage[english]{babel}
\usepackage{amsfonts}
\usepackage[11pt]{moresize}
\usepackage{pgfplots}
\usepackage{xfrac}

\DeclareMathOperator{\Gr}{Gr}

\usepackage{amscd}
\usepackage[breaklinks,bookmarksopen,bookmarksnumbered]{hyperref}
\usepackage{amssymb}
\usepackage{mathtools}
\usepackage{tikz}
\usepackage{tikz-cd}
\usepackage{mathtools}
\usepackage{amsmath}
\usepackage{amsthm}
\usepackage{hyperref}
\usepackage{float}
\usetikzlibrary{calc, positioning,decorations.markings,arrows}
\usepackage{graphics}
\usepackage{color}
\usepackage{tabularx,colortbl}

\newcommand{\Addresses}{{% additional braces for segregating \footnotesize
		\bigskip
		\footnotesize
		
		Humboldt-Universit\"at zu Berlin, Institut f\"ur Mathematik, Rudower Chausee 25
		\hfill \newline\texttt{}
		\indent 12489 Berlin, Germany} 
	\par\nopagebreak
	\textit{E-mail address}: \texttt{andreibud95@protonmail.com}
}

\theoremstyle{plain}
\newtheorem{trm}{Theorem}[section]

\newtheorem{prop}[trm]{Proposition}

\theoremstyle{definition}

\newtheorem{rmk}[trm]{Remark}
\newtheorem{conj}[trm]{Conjecture}

\begin{document}
	\title{Prym enumerative geometry and a Hurwitz divisor in $\overline{\mathcal{R}}_{2i}$} 
	\author{Andrei Bud}
	\begin{abstract}  For $i\geq2$, we compute the first coefficients of the class $[\overline{D}(\mu;3)]$ in the rational Picard group of the moduli of Prym curves $\overline{\mathcal{R}}_{2i}$, where $D(\mu;3)$ is the divisor parametrizing pairs $[C,\eta]$ for which there exists a degree $2i$ map $\pi\colon C\rightarrow \mathbb{P}^1$ having ramification profile $(2,\ldots,2)$ above two points $q_1, q_2$, a triple ramification somewhere else and satisfying $\mathcal{O}_C(\frac{\pi^{*}(q_1)-\pi^{*}(q_2)}{2})\cong \eta$. Furthermore, we provide several new Prym enumerative results related to this situation. 
	\end{abstract}
\maketitle
\section{Introduction} \label{sec1}   

The moduli space $\mathcal{R}_g$ of smooth Prym curves, parametrizing pairs $[C,\eta]$ where $[C]\in \mathcal{M}_g$ and $\eta$ is a 2-torsion line bundle on $C$, received considerable attention following the influential papers \cite{MumfordPrym} and \cite{Beau77} of Mumford and Beauville. The algebraic theory of Prym curves developed by Mumford, together with the modular interpretation of $\mathcal{R}_g$ provided by Beauville laid the foundation for an algebraic geometric study of Prym curves. Through this perspective, the associated map $\mathcal{P}_g\colon\mathcal{R}_g\rightarrow \mathcal{A}_{g-1}$ to the moduli space of principally polarized Abelian varieties was used to provide an algebraic proof of the Schottky-Jung relations, see \cite{MumfordPrym}, and to understand the birational geometry of the moduli of Prym varieties, see \cite{FarLud}, \cite{Bruns}, \cite{FarVerNikulin}, \cite{FarR13} and the references therein.  

An important problem regarding this moduli space is understanding its birational geometry. As such, we recognize the important role played by Hurwitz divisors in proving that $\mathcal{M}_g$ is of general type when $g\geq 24$, see \cite{KodMg}, \cite{KodevenHarris1984} , \cite{EisenbudHarrisg>23} and that $\mathcal{R}_g$ is of general type when $g\geq 13, g\neq 16$, see \cite{FarLud},\cite{Bruns} and \cite{FarR13}. Along with those, several other Hurwitz divisors appear in the literature in \cite{Diaz}, \cite{VGeerKouvid}, \cite{FarkasFermat} and \cite{Bud-adm}. Our goal is to consider a new Hurwitz divisor $\overline{D}(\mu;3) \subseteq \overline{\mathcal{R}}_{2i}$ and compute some of its coefficients in the Picard group $\mathrm{Pic}_\mathbb{Q}(\overline{\mathcal{R}}_{2i})$, where $\overline{\mathcal{R}}_{2i}$ denotes the compactification of $\mathcal{R}_{2i}$ appearing in \cite{Casa} and \cite{FarLud}. 

 To compute the coefficients of our divisor, we will require several new enumerative results, counting pencils with some given ramification profiles on a generic Prym curve. As such, we first extend on the work done in \cite{Bud-adm} in order to provide such results for elliptic curves. Furthermore, for a generic Prym curve $[C,\eta]$ in $\mathcal{R}_g$, we will provide some Prym analogues of the enumerative results in \cite{KodMg} and \cite{KodevenHarris1984}. Drawing a parallel with the situation for $\overline{\mathcal{M}}_g$ where a plethora of enumerative results is required to compute the classes of Hurwitz divisors, we expect our results to be useful in the study of cycles on $\overline{\mathcal{R}}_g$.
%An important role in understanding the birational type of the moduli spaces $\mathcal{M}_g$ and $\mathcal{R}_g$ is played by Hurwitz divisors. Such divisors are essential in proving that $\mathcal{M}_g$ is of general type when $g\geq 24$, see \cite{KodMg}, \cite{KodevenHarris1984} and \cite{EisenbudHarrisg>23} and in proving that $\mathcal{R}_g$ is of general type when $g\geq 17$, see \cite{FarLud}. Along with those, several other Hurwitz divisors appear in the literature in \cite{Diaz}, \cite{VGeerKouvid}, \cite{FarkasFermat} and \cite{Bud-adm}. The goal of this paper is to provide a new such divisor on the moduli space $\overline{\mathcal{R}}_{2i}$ for $i\geq 2$ and to provide some enumerative results for pencils over a generic element $[C,\eta] \in \mathcal{R}_g$.   
%The divisors appearing as push-forwards of Hurwitz spaces to different moduli spaces play an important role in understanding their geometry. Such push-forward divisors are essential in proving that $\mathcal{M}_g$ is of general type when $g\geq 24$, see \cite{KodMg}, \cite{KodevenHarris1984} and \cite{EisenbudHarrisg>23} and in proving that $\mathcal{R}_g$ is of general type when $g\geq 17$, see \cite{FarLud}. Along with those, several other such divisors appear in the literature in \cite{VGeerKouvid}, \cite{FarkasFermat} and \cite{Bud3}. The goal of this paper is to provide a new such divisor on the moduli space $\overline{\mathcal{R}}_{2i}$ for $i\geq 2$ and to provide some enumerative results for pencils on a generic element $[C,\eta] \in \mathcal{R}_g$.   
	
For $g=2i$ and the length $2i$ partition $\mu = (2,2,\ldots,2, -2,-2,\ldots,-2)$ of 0, we consider $H_{\mu;3}$ to be the Hurwitz scheme parametrizing up to isomorphism degree $2i$ maps $\pi\colon X \rightarrow \mathbb{P}^1$ having ramification profiles $(2,\ldots, 2)$, $(2,\ldots,2)$ and $(3,1,1, \ldots,1)$ over three branch points $q_1$, $q_2$ and $q_3$ and is otherwise simply ramified. This scheme admits a compactification using admissible covers, see \cite{KodMg}, \cite{Diaz} and \cite{AbramovichCV}, which we denote $\overline{H}_{\mu;3}$. We can define a map 
\[ c_{\mu;3} \colon H_{\mu;3} \rightarrow \overline{\mathcal{R}}_{2i}\] 
sending $[\pi\colon X \rightarrow \Gamma]$ to $[X, \mathcal{O}_X(\frac{\pi^{*}(q_1)-\pi^{*}(q_2)}{2})]$. Using the existence and unicity of a twist as in \cite{FP18}, this map can be extended over admissible covers $[\pi\colon X \rightarrow \Gamma]$ where $X$ is a curve of compact type. As in \cite{Bud-adm}, we can extend this over admissible covers $[\pi\colon X \rightarrow \Gamma]$, where $X$ stabilizes to a curve $[C/{x\sim y}]$ with $[C,x]\in \mathcal{M}_{2i-1,1}$ generic. 

The image of the map $c_{\mu;3}$ is a divisor, which we denote by $\overline{D}(\mu;3)$. We compute some of its coefficients in $\mathrm{Pic}_{\mathbb{Q}}(\overline{\mathcal{R}}_{2i})$ with respect to the standard basis formed by $\lambda$ and the classes of the boundary divisors: 
\begin{trm}\label{divisor} The first coefficients of the class of $\overline{D}(\mu;3)$ in $\mathrm{Pic}_\mathbb{Q}(\overline{\mathcal{R}}_{2i})$ are given by the formula 
\[ \overline{D}(\mu;3) \equiv (6i-4)!\cdot\binom{2i-1}{i}\cdot\left(a\lambda-b_0'\delta_0'-b_0''\delta_0''-b_0^{\mathrm{ram}}\delta_0^{\mathrm{ram}}-b_1\delta_1-b_{g-1}\delta_{g-1}-b_{1:g-1}\delta_{1: g-1}-\cdots\right)\]
where $b_1 = 48i^2-4i+4$, $b_{g-1} = 36i+12$, $b_{1:g-1} =  12i+12$, $a=36i+96+\frac{36}{2i-1}$ and moreover 
\[b_0' = 6i+9+\frac{3}{2i-1}, \ \ b_0'' = 12i^2+2i+7+\frac{3}{2i-1}, \ \ b_0^{\mathrm{ram}}= 6i+18+\frac{6}{2i-1} \]
	
\end{trm}

The proof of Theorem \ref{divisor} is via intersection with test curves and as such, we will require several enumerative results. For this, we will study in Section \ref{ellipticsection} the enumerative geometry of a general elliptic curve. Furthermore, we will provide in Proposition \ref{tripleram}, Proposition \ref{2prym} and Proposition \ref{prymunram} several enumerative results computing the number of maps to $\mathbb{P}^1$ satisfying certain ramification conditions. Regarding such enumerative results on $\overline{\mathcal{M}}_g$, the case where above each branch point there is a unique point of ramification was fully solved in \cite{Lian19}. We will provide in Proposition \ref{enumerativeMg} a count for a case where there are two ramification points above a branch point. 

We recall that there is a bijective correspondence between smooth Prym curves $[C,\eta]$ and \'etale double covers $[\widetilde{C}\rightarrow C]$. This correspondence allows us to define a map $\chi_g\colon\mathcal{R}_g$ $\rightarrow$ $\mathcal{M}_{2g-1}$ sending $[C,\eta]$ to the associated $\widetilde{C}$. The pencils we count in Proposition \ref{tripleram} can be used to show that, when $g=2i-1$, the image $\mathrm{Im}(\chi_g)$ is contained in the Hurwitz divisor $\mathfrak{TR}_{2i}$ studied by Farkas in \cite{FarkasFermat}. In fact, our approach can be adapted as to show that $\mathrm{Im}(\chi_g)$ sits on many other Hurwitz divisors, for both odd and even $g$. 

Finally in Section \ref{sec-intwithtestcurv} we describe the test curves we are going to use, while in Section \ref{finalsection} we compute the required intersections. The proof of Theorem \ref*{divisor} is an immediate consequence of these computations.

\textbf{Acknowledgements:} I am grateful to my advisor Gavril Farkas and to Johannes Schmitt for their insights. I am also thankful to Carlos Maestro P\'erez and Carl Lian for the helpful conversations on the topic of this paper.

\section{Enumerative geometry on the generic elliptic curve} \label{ellipticsection}

To prove Theorem \ref{divisor}, we will intersect the divisor $\overline{D}(\mu;3)$ with various test curves. We will provide the necessary enumerative results that will help us understand the admissible covers in $\overline{H}_{\mu;3}$ appearing above such intersections. In this section we will concentrate on the enumerative geometry of a generic pointed elliptic curve. We start with a result that is similar to Proposition 2.5 and Proposition 2.6 in \cite{Bud-adm}.
\subsection{A Hurwitz space over elliptic curves}
Let $k\geq 2$ and let $b_1 = (2k-1,1)$, $b_2 = b_3 = (2,\ldots,2)$ and $b_4 = (3,1,\ldots,1)$ be partitions of $2k$. We denote $B = \left\{b_1,b_2,b_3,b_4\right\}$ and consider $\overline{H}_{2k,B}$ the Hurwitz scheme parametrizing admissible covers $\pi\colon X \rightarrow \Gamma$ having ramification profile $b_i$ over a point $q_i$ for $i = \overline{1,4}$ and otherwise unramified. We consider the map 
\[ \pi_k\colon \overline{H}_{2k,B}\rightarrow \overline{\mathcal{M}}_{1,1}\] 
remembering only the point of order $2k-1$ over $q_1$ (and stabilizing the pointed source curve if necessary). 

\begin{prop} \label{elliptic}The degree of $\pi_k$ is $12$. 
\end{prop}
\begin{proof} We proceed as in \cite{Bud-adm} Propositions 2.5 and 2.6. We consider the singular pointed curve $[E_\infty,p]$ of $\overline{\mathcal{M}}_{1,1}$ and we compute the length of $\pi_k^{*}([E_\infty,p])$, which we know is equal to the degree of the map.
	
Let $\pi\colon X\rightarrow \Gamma$ be an admissible cover mapped by $\pi_k$ to $[E_\infty,p]$. We denote by $R$ the rational component of $X$ mapping to $E_\infty$ and by $R_1$ the component collapsing to the node of the curve $E_\infty$. We denote by $u$ and $v$ the two nodes where $R$ and $R_1$ are glued together. It follows that $R$ contains the point $p$ of ramification order $2k-1$ over $q_1$. For the target curve $\Gamma$, we denote by $\mathbb{P}_1$ the target of $R$, by $\mathbb{P}_2$ the target of $R_1$ and by $q$ the node. Finally, we denote by $f$ and $f_1$ the restriction of $\pi$ to $R$ and $R_1$ respectively. Depending on the position of the points $q_1,\ldots,q_4$ we distinguish three different cases. 

\textbf{Case I:} The points $q_1$ and $q_4$ are on $\mathbb{P}_1$. In this case, the degree of $f$ is either $2k-1$ or $2k$. 

If $\mathrm{deg}(f) = 2k-1$, the Riemann-Hurwitz theorem implies that there are $3$ points on $R$ over $q$. It follows that the curve $X$ is of the form 

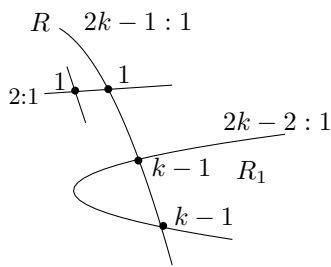
\begin{figure}[H]\centering 
	\begin{tikzpicture}
	%the curve C 
	\draw [domain=0.3:1.8] plot ({\x},{sqrt{4*\x-\x*\x} - 2*\x});
	%the curve E
	\draw [domain=-2:5] plot ({\x*\x/5-\x/2+0.8)},{\x/5-2.5});
	%the rational curves glued to C
	%\draw [domain=0.6:2.3] plot ({\x}, {\x/15-2});
	\draw [domain=0.6:2.3] plot ({\x-0.5}, {\x/15-1}); 
	%the arrow
	%\draw [decoration={markings,mark=at position 1 with
	%{\arrow[scale=2,>=stealth]{>}}},postaction={decorate}] (1.5,-4) -- (1.5,-5);
	%P2
	%\draw [domain=0.7:2.8] plot ({\x}, {\x/15-6}); 
	%P1
	%\draw [domain=1:1.5] plot ({\x}, {-3*\x-2});
	
	%the P^1's over P_1
	%\draw [domain=1.2:1.45] plot ({\x-0.5}, {-3*\x+2});
	\draw [domain=1.2:1.45] plot ({\x-0.8}, {-3*\x+3});
	%names
	%\node[left = 1mm of {(3.3,-3)}] {$E$};
	\node[left = 1mm of {(0.4,0)}] {$R$};
	\node[right = 1mm of {(0.4,0)}] {$2k-1:1$};
	\node[left = 1mm of {(3.3,-2)}] {$R_1$};
	\node[above = 1mm of {(3.2,-1.7)}] {$2k-2:1$};
	%\node[left = 1mm of {(1.3,-5.3)}] {$\mathbb{P}_1$};
	%\node[left = 1mm of {(3,-6.2)}] {$\mathbb{P}_2$};
	%point q
	%\node[mark size=1.2pt,color=black] at (1.3,-5.91) {\pgfuseplotmark{*}};
	%\node[right = 1mm of {(1.2,-5.7)}] {$q$};
	%glueing points
	\fill (0.95,-0.9) circle[radius=1.5pt];
	\fill (0.51,-0.92) circle[radius=1.5pt];
	\node[left = 1mm of {(0.62,-0.8)}] {$1$};
	%\node[mark size=1.2pt,color=black] at (1.37,-1.9) {\pgfuseplotmark{*}};
	%point x
	\fill (1.68,-2.72) circle[radius=1.5pt];
	\node[right = 1mm of {(1.6,-2.6)}] {$k-1$};
	%point y 
	\fill (1.35,-1.85) circle[radius=1.5pt];
	\node[right = 1mm of {(1.3,-1.95)}] {$k-1$};
	\node[right = 1mm of { (0.85,-0.7)}] {$1$};
	%arrows
	%\draw [->]  (2.2, -1.4) -- (1.5,-1.8) ;
	%\draw [->]  (2.2, -1.4) -- (1.1,-1) ;
	%description
	\node[right = 1mm of {(-0.6,-1)}] {$\text{\small 2:1}$};
	%\node[right = 1mm of {(0,-1.7)}] {$\text{\small 1:1}$};
	%\node[right = 1mm of {(2.2,-1.4)}] {Points in the same fiber as $x$};
	%\node[right = 1mm of {(0.4,-3)}] {$2:1$};
	%\node[right = 1mm of {(2.2,-1.8)}] {$2:1$};
	
\end{tikzpicture} \caption{A curve $X$ corresponding to an admissible cover in $\overline{H}_{2k,B}$ over $[E_\infty,p]$ }
\end{figure} 
Here $f$ is $2k-1$ to $1$ and has ramification profiles $(2k-1), (3,1,\ldots,1)$ and $(k-1,k-1,1)$ over $q_1,q_4$ and $q$. The same combinatorial argument as in Proposition 2.4 in \cite{Bud-adm} implies that the map $f$ is unique (up to automorphisms of $\mathbb{P}^1$). The map $f_1$ is also unique, $2k-2$ to $1$, with ramification profiles $(2,2,\ldots,2)$, $(2,\ldots,2)$ and $(k-1,k-1)$ over the points $q_2, q_3$ and $q$. Such a map was described in \cite{Bud-adm}, proof of Proposition 2.5 and we know it has $2k-2$ automorphisms. Looking at the local description of the Hurwitz scheme in the neighbourhood of the admissible cover $\pi\colon X \rightarrow \mathbb{P}^1$ we deduce that it should be counted with multiplicity $2$. 

If $\mathrm{deg}(f) = 2k$, the only points above $q$ are $u$ and $v$. The table of Proposition 2.4 in \cite{Bud-adm} implies that both $u$ and $v$ are points of ramification order $k$. In this case $f$ has ramification profiles $(2k-1,1)$, $(3,1,\ldots,1)$ and $(k,k)$ and a combinatorial argument would imply that it is unique up to an automorphism of $\mathbb{P}^1$. The map $f_1$ is again unique,  with ramification profiles $(2,2,\ldots,2)$, $(2,\ldots,2)$ and $(k,k)$ over the points $q_2, q_3$ and $q$. We get again that such an admissible cover appears with multiplicity $2$. 

As a consequence, the contribution to the length coming from \textbf{Case I} is $2+2 = 4$. 

\textbf{Case II:} The points $q_1$ and $q_2$ are on $\mathbb{P}_1$. This implies $\mathrm{deg}(f) = 2k$. The Riemann-Hurwitz theorem implies there are $k$ points on $R$ over $q$. Hence we get that $X$ is of the form 
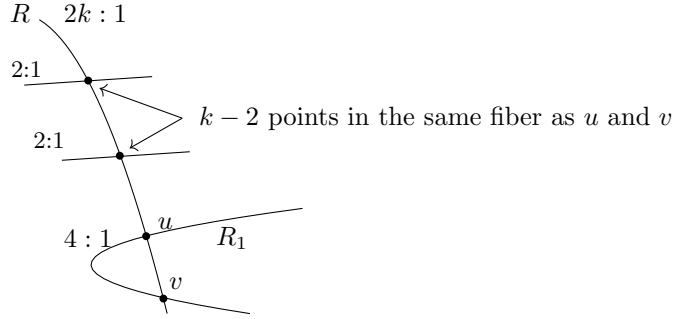
\begin{figure}[H] \centering 
	\begin{tikzpicture}
	%the curve C 
	\draw [domain=0.3:2] plot ({\x},{sqrt{4*\x-\x*\x} - 2*\x});
	%the curve E
	\draw [domain=-2:5] plot ({\x*\x/5-\x/2+1.3)},{\x/5-3.6});
	%the rational curves glued to C
	\draw [domain=0.6:2.3] plot ({\x}, {\x/15-2});
	\draw [domain=0.6:2.3] plot ({\x-0.5}, {\x/15-1}); 
	%the arrow
	%\draw [decoration={markings,mark=at position 1 with
	%{\arrow[scale=2,>=stealth]{>}}},postaction={decorate}] (1.5,-4) -- (1.5,-5);
	%P2
	%\draw [domain=0.7:2.8] plot ({\x}, {\x/15-6}); 
	%P1
	%\draw [domain=1:1.5] plot ({\x}, {-3*\x-2});
	
	%the P^1's over P_1
	%\draw [domain=1.2:1.45] plot ({\x-0.5}, {-3*\x+2});
	%\draw [domain=1.2:1.45] plot ({\x-0.8}, {-3*\x+3});
	%names
	%\node[left = 1mm of {(3.3,-3)}] {$E$};
	\node[left = 1mm of {(0.4,0)}] {$R$};
	\node[right = 1mm of {(0.4,0)}] {$2k:1$};
	\node[left = 1mm of {(3.3,-3)}] {$R_1$};
	%\node[left = 1mm of {(1.3,-5.3)}] {$\mathbb{P}_1$};
	%\node[left = 1mm of {(3,-6.2)}] {$\mathbb{P}_2$};
	%point q
	%\node[mark size=1.2pt,color=black] at (1.3,-5.91) {\pgfuseplotmark{*}};
	%\node[right = 1mm of {(1.2,-5.7)}] {$q$};
	
	%glueing points
	\fill (0.95,-0.9)  circle[radius=1.5pt];
	\fill (1.37,-1.9) circle[radius=1.5pt];
	%point x
	\fill (1.72,-2.97) circle[radius=1.5pt];
	\node[right = 1mm of {(1.65,-2.8)}] {$u$};
	%point y 
	\fill (1.95,-3.8)  circle[radius=1.5pt];
	\node[right = 1mm of {(1.8,-3.6)}] {$v$};
	%arrows
	\draw [->]  (2.2, -1.4) -- (1.5,-1.8) ;
	\draw [->]  (2.2, -1.4) -- (1.1,-1) ;
	%description
	\node[right = 1mm of {(-0.3,-0.75)}] {$\text{\small 2:1}$};
	\node[right = 1mm of {(0,-1.7)}] {$\text{\small 2:1}$};
	\node[right = 1mm of {(2.2,-1.4)}] {$k-2$ points in the same fiber as $u$ and $v$};
	\node[right = 1mm of {(0.4,-3)}] {$4:1$};
	%\node[right = 1mm of {(2.2,-1.8)}] {$2:1$};
	\end{tikzpicture} \caption{A curve $X$ corresponding to an admissible cover in Case II }
\end{figure}
The map $f$ has ramification profiles $(2k-1,1)$, $(2,\ldots,2)$ and $(1,3,2,\ldots,2)$ over $q_1, q_2$ and $q$, with $u$ and $v$ the points of orders $1$ and $3$. Proposition 2.4 in \cite{Bud-adm} implies that such a map is unique. Similarly the map $f_1$ having ramification profiles $(3,1), (3,1)$ and $(2,2)$ is unique. It is immediately checked that such an admissible cover should be counted with multiplicity $1+3 = 4$. 

As a consequence, the contribution from \textbf{Case II} is $4$. 
The third case when $q_1$ and $q_3$ are on $\mathbb{P}_1$ is identical to the second one, and hence we get another contribution of $4$. 

Consequently $\mathrm{deg}(\pi_k) = 12$.  
\end{proof}
The same method can be used to prove the following result for $2$-pointed elliptic curves:
\begin{prop} \label{elliptic2points}
	Let $[E,y,z] \in \mathcal{M}_{1,2}$ generic. The number of degree $2i-2s$ maps $\pi\colon E \rightarrow \mathbb{P}^1$ having ramification profiles $b_1 = (2i-2s-1,1), b_2 = b_3 = (2,\ldots,2)$ and $b_4 = b_5 = (2,1,\ldots,1)$ over $q_1,\ldots,q_5$ with $y$ and $z$ the ramified points over $q_1$ and $q_4$ is equal to $6$ if $s=i-1$ and $12$ when $s<i-1$. 
\end{prop} 

\begin{proof} This proposition is clear for $s=i-1$. When $s< i-1$ we take $\overline{H}_{2i-2s, B}$ the Hurwitz scheme parametrizing admissible covers with ramification profiles given by $B= \left\{b_1,\ldots, b_5\right\}$. We take the map 
\[ \pi_{2i-2s, B} \colon \overline{H}_{2i-2s, B}\rightarrow \overline{\mathcal{M}}_{1,2}\]
remembering the points of ramification order $2i-2s-1$ and the ramification point over $q_4$. To compute the degree of this map, we look at its fiber above a point $[E\cup_{x}\mathbb{P}^1,y,z]$ where $[E,x] \in \mathcal{M}_{1,1}$ is generic and $y,z \in \mathbb{P}^1$. We have two possibilities for an admissible cover $\pi\colon X\rightarrow \Gamma$ above such a point. 

1. The ramification order at the node is $2i-2s$. Then the restriction of $\pi$ to the rational component is the unique map having ramification profiles $(2i-2s)$, $(2i-2s-1,1)$,$(2,1,\ldots,1)$ above $q, q_1, q_4$, while the restriction to the elliptic component has degree $2i-2s$ and ramification profiles $(2i-2s), (2,\ldots,2), (2,\ldots,2)$ and $(2,1,\ldots,1)$. 

2. The ramification order at the node is $2i-2s-2$. Then the restriction of $\pi$ is the unique map of degree $2i-2s-1$ having ramification profiles $(2i-2s-1)$, $(2,1,\ldots,1)$, $(2i-2s-2,1)$ above $q_1, q_4$ and $q$. Moreover the restriction to the elliptic component has degree $2i-2s-2$ and ramification profiles $(2i-2s-2), (2,\ldots,2), (2,\ldots,2)$ and $(2,1,\ldots,1)$. 

Because all admissible covers appear with multiplicity $1$, we conclude using Proposition 2.5 in \cite{Bud-adm} that $\mathrm{deg}(\pi_{2i-2s, B}) = 12$.
\end{proof}

Lastly, we have:
\begin{prop} \label{ellipticunramified2pts}
	Let $[E,y,z] \in \mathcal{M}_{1,2}$ generic and $0\leq s\leq i-2$. The number of degree $2i-2s$ maps $\pi\colon E \rightarrow \mathbb{P}^1$ with ramification profiles $b_1 = (2i-2s-1,1), b_2 = b_3 = (2,\ldots, 2)$, $b_4=b_5=(2,1,\ldots,1)$ over $q_1,\ldots,q_5$ with $y$ the ramified point over $q_1$ and $z$ an unramified point over $q_4$ is equal to $12(2i-2s-1)$.
\end{prop} 
\begin{proof}
  We consider the Hurwitz stack $\overline{H}_{2i-2s,B,1}$ parametrizing the pairs of an admissible cover $\pi\colon E \rightarrow \mathbb{P}^1$ having ramification profiles $b_1,\ldots, b_5$ over $q_1,\ldots, q_5$, together with an unramified point $x$ over $q_4$. We consider the map 
\[ \pi_{2i-2s, B,1}\colon \overline{H}_{2i-2s, B,1} \rightarrow \overline{\mathcal{M}}_{1,2}\]
remembering only the ramified point over $q_1$ and the point $x$. 

We consider again $[E\cup_{x}\mathbb{P}^1,y,z]$ and we compute the length of $\pi^{*}_{2i-2s,B,1}([E\cup_{x}\mathbb{P}^1,y,z])$. There are two possibilities for an admissible cover above such a point: either $z$ is on the rational component that does not collapse when we stabilize, or $z$ is on a rational component that gets collapsed in the stabilization. For the first possibility we get a contribution of $6\cdot(4i-4s-5)$ to the length of the cycle, while for the second one we get a contribution of $6\cdot3$.
Adding up the contributions we get the desired conclusion.
\end{proof}
\section{Enumerative geometry on the generic curve} \label{genericcurvesection}
\begin{prop} \label{2i-2s,3}Let $[C] \in \mathcal{M}_{2i-1}$ a generic curve. The number of pencils $L \in W^1_{2i-s}(C)$ satisfying the conditions 
\[ h^0(C, L(-3x)) = h^0(C, L(-(2i-2s)y)) = 1 \] for some distinct points $x, y \in C$ is equal to
\[ 8(i-s-1)(2i-2s-1)(2i-2s+1)\binom{2i-1}{s} + 8 (i-s+1)(2i-2s-1)(2i-2s+1)\binom{2i-1}{s-1} + \cdots\]
\[\cdots +8(s-1)(2i-2s+2)(2i-2s-1)(2i-2s+1)\binom{2i-1}{s-1} + 8s(2i-2s)(2i-2s-1)(2i-2s+1)\binom{2i-1}{s}+\cdots\]
\[\cdots + 8(s+1)(2i-2s-2)(2i-2s-1)(2i-2s+1)\binom{2i-1}{s+1}\]
\end{prop} 
\begin{proof}
	We will apply Proposition 5.4 in \cite{Lian19} to our particular case. Adopting the conventions in \cite{Lian19}, we distinguish two possible cases: 
	
	\textbf{1.} The points $x$ and $y$ are in the same box of the distribution; or 
	
	\textbf{2.} The points $x$ and $y$ are in different boxes of the distribution. 
	
	If we are in Case \textbf{1.} and the "special" box is the first one, we get that the possible vanishing sequences $(a_1,b_1)$ satisfying 
	\[a_1+b_1+2+3+2i-2s = 4i-2s+4\] 
	are $(s-1, 2i-s)$ and $(s,2i-s-1)$. For all the other $g-1$ boxes we have $a_j + b_j = 4i-2s-2$ and we get the unique possible vanishing sequence $(2i-s-2,2i-s)$. In this case we get a contribution of 
	\begin{equation*}
	\left(\int_{\Gr(2,2i-s+1)}\sigma_1^{g-1}\cdot\sigma_{2i-2s}\right)\cdot N_{2,2,2,2}^{g-1}\cdot N_{2i-2s+1,2i-2s,3,2} +\cdots\end{equation*}
	\begin{equation*}
	\cdots+\left(\int_{\Gr(2,2i-s+1)}\sigma_1^{g-1}\cdot\sigma_{2i-2s-1,1}\right)\cdot N_{2,2,2,2}^{g-1}\cdot N_{2i-2s-1,2i-2s,3,2}
	\end{equation*}
	It is clear that $N_{2,2,2,2} = 6$. We know from \cite{KodevenHarris1984} that $N_{2i-2s+1,2i-2s,3,2} = 8(i-s+1)(2i-2s-1)$ and $N_{2i-2s-1,2i-2s,3,2} = 8(i-s-1)(2i-2s+1)$. Using Example 14.7.11 in \cite{Fulton-Intersection} we have the following formula for Schubert intersections: 
	\[ \int_{\Gr(2,d+1)}\sigma_1^{2d-2-\lambda_0-\lambda_1}\cdot\sigma_{\lambda_0,\lambda_1} = \frac{(2d-2-\lambda_0-\lambda_1)!}{(d-1-\lambda_0)!\cdot (d-\lambda_1)!}\cdot(\lambda_0+1-\lambda_1)  \] 
	Applied to our situation, this gives 
	\[ \int_{\Gr(2,2i-s+1)}\sigma_1^{2i-2}\cdot\sigma_{2i-2s} = \frac{(2i-2)!}{(s-1)!\cdot (2i-s)!}\cdot(2i-2s+1) = \frac{2i-2s+1}{2i-1} \binom{2i-1}{s-1}  \] 
	and similarly 
	\[ \int_{\Gr(2,2i-s+1)}\sigma_1^{2i-2}\cdot\sigma_{2i-2s-1,1} = \frac{(2i-2)!}{(s)!\cdot (2i-s-1)!}\cdot(2i-2s-1) = \frac{2i-2s-1}{2i-1}\binom{2i-1}{s}  \]
	We get a contribution of 
	\[ \frac{2i-2s+1}{2i-1} \binom{2i-1}{s-1}\cdot 6^{g-1}\cdot 8(i-s+1)(2i-2s-1) + \frac{2i-2s-1}{2i-1}\binom{2i-1}{s} \cdot 6^{g-1} \cdot 8(i-s-1)(2i-2s+1) \] 
	But the number of distributions having $x$ and $y$ in the same box is equal to $(2i-1)\frac{(3g-2)!}{6^{g-1}}$. Moreover, the order of the $(3g-2)$ simple ramification points is irrelevant to us. We hence get a contribution to the count equal to 
	\[ 8 (i-s+1)(2i-2s-1)(2i-2s+1)\binom{2i-1}{s-1} + 8(i-s-1)(2i-2s-1)(2i-2s+1)\binom{2i-1}{s}\]
	
	We now compute the contribution coming from Case \textbf{2.} when $x$ and $y$ are in different boxes, which we assume to be the first and the second one. In this case, the only possible vanishing sequences are $(a_1, b_1) = (s,2i-s)$ and $(a_2,b_2) = (2i-s-3, 2i-s)$. 
	
	As $N_{k,k,2,2} = 2(k^2-1)$, we get the contribution in this case to be 
	\begin{equation*}
	\left(\int_{\Gr(2,2i-s+1)}\sigma_1^{g-2}\cdot\sigma_2\cdot\sigma_{2i-2s-1}\right)\cdot 6^{g-2}\cdot 16\cdot 2[4(i-s)^2-1] \end{equation*}
	Pieri's rule gives the equality 
	\[ \sigma_2\cdot\sigma_{2i-2s-1} = \sigma_{2i-2s+1} + \sigma_{2i-2s,1}+ \sigma_{2i-2s-1,2}\] 
	We compute 
	\[ \int_{\Gr(2,2i-s+1)}\sigma_1^{2i-3}\cdot\sigma_{2i-2s+1} = \frac{(2i-3)!}{(s-2)!\cdot (2i-s)!}\cdot(2i-2s+2)   \] 
	\[ \int_{\Gr(2,2i-s+1)}\sigma_1^{2i-3}\cdot\sigma_{2i-2s,1} = \frac{(2i-3)!}{(s-1)!\cdot (2i-s-1)!}\cdot(2i-2s)   \]
	\[ \int_{\Gr(2,2i-s+1)}\sigma_1^{2i-3}\cdot\sigma_{2i-2s-1,2} = \frac{(2i-3)!}{s!\cdot (2i-s-2)!}\cdot(2i-2s-2)   \] 
The number of distributions with $x$ and $y$ in different boxes is $\frac{(3g-2)!}{4\cdot6^{g-2}}\cdot g(g-1)$ and we factor out the term $(3g-2)!$ corresponding to the order of the $(3g-2)$ simple ramifications. The contribution we get in this case corresponds to the sum of the last three terms in the proposition.
\end{proof} 

\begin{rmk} When $s = i-1$, the formula gives the known answer $24(6i-4) \binom{2i-1}{i+1}$. 
\end{rmk}

\begin{prop} \label{tripleram}  Let $g=2i-1$ and $[X,\eta] \in \mathcal{R}_{g}$ a generic point. We consider the degree $2i$ maps $\pi\colon X\rightarrow \mathbb{P}^1$ such that the ramification profiles over three branch points $q_1$, $q_2$ and $q_3$ are $(2,\ldots,2), (2,\ldots,2)$ and $(3,1,1,\ldots,1)$ respectively and all other branch points are simple. The number of such maps satisfying 
\[ \mathcal{O}_X(\frac{\pi^{*}(q_1)-\pi^{*}(q_2)}{2}) \cong \eta \]
is equal to $24(i-1)\binom{2i-1}{i}$. 
\end{prop} 
\begin{proof}
	We consider the Hurwitz scheme $\overline{H}_{2i,\mu,3}$ parametrizing degree $2i$ maps having ramification profiles $(2,\ldots,2)$, $(2,\ldots,2)$ and $(3,1,\ldots,1)$ over three branch points $q_1, q_2$ and $q_3$, while all other ramifications are simple. 
	We have a rational map 
	\[ c_{\mu;3}\colon \overline{H}_{2i,\mu,3} \rightarrow \overline{\mathcal{R}}_g \] 
	sending $[\pi\colon X \rightarrow \mathbb{P}^1]$ to $[X, \mathcal{O}_X(\frac{\pi^{*}(q_1)-\pi^{*}(q_2)}{2})]$ if $X$ is smooth. The existence of a twist as in \cite{FP18} implies that the map can be extended to admissible covers with source curve of compact type, see \cite{Bud-adm}. 
	
	We want to prove that the degree of $c_{\mu;3}$ is $24(i-1)\binom{2i-1}{i}$. For this we consider $[C\cup_{x\sim y} E, \mathcal{O}_C, \eta_E]$ a generic point of the boundary divisor $\Delta_1$ and compute the length of the cycle $c_{\mu;3}^{*}([C\cup_{x\sim y} E, \mathcal{O}_C, \eta_E])$. The admissible covers $\pi\colon X\rightarrow \Gamma$ above such a point are of the following form:
	
		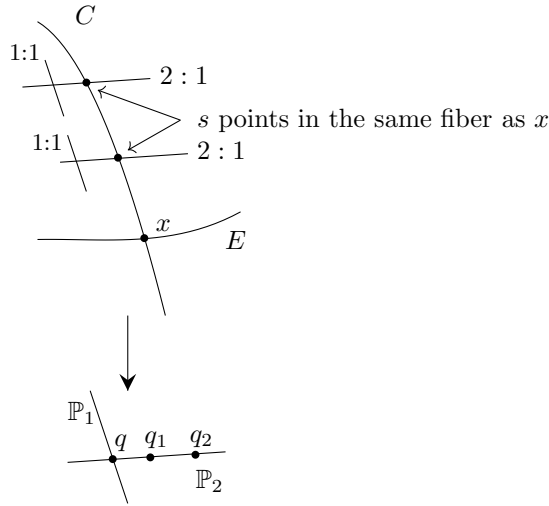
\begin{figure}[H] \centering 
		\begin{tikzpicture}
		%the curve C 
		\draw [domain=0.3:2] plot ({\x},{sqrt{4*\x-\x*\x} - 2*\x});
		%the curve E
		\draw [domain=0.3:3] plot ({\x},{\x/15 - \x*\x/10 + \x*\x*\x/25 - 3});
		%the rational curves glued to C
		\draw [domain=0.6:2.3] plot ({\x}, {\x/15-2});
		\draw [domain=0.6:2.3] plot ({\x-0.5}, {\x/15-1}); 
		%the arrow
		\draw [decoration={markings,mark=at position 1 with
			{\arrow[scale=2,>=stealth]{>}}},postaction={decorate}] (1.5,-4) -- (1.5,-5);
		%P2
		\draw [domain=0.7:2.8] plot ({\x}, {\x/15-6}); 
		%P1
		\draw [domain=1:1.5] plot ({\x}, {-3*\x-2});
		
		%the P^1's over P_1
		\draw [domain=1.2:1.45] plot ({\x-0.5}, {-3*\x+2});
		\draw [domain=1.2:1.45] plot ({\x-0.8}, {-3*\x+3});
		%names
		\node[left = 1mm of {(3.3,-3)}] {$E$};
		\node[left = 1mm of {(1.3,0)}] {$C$};
		\node[left = 1mm of {(1.3,-5.3)}] {$\mathbb{P}_1$};
		\node[left = 1mm of {(3,-6.2)}] {$\mathbb{P}_2$};
		%point q
		\fill (1.3,-5.91) circle[radius=1.5pt];
		\node[right = 1mm of {(1.1,-5.7)}] {$q$};
		
		\fill (1.8,-5.887)  circle[radius=1.5pt];
	    \node[right = 1mm of {(1.5,-5.66)}] {$q_1$};
	    
	    \fill (2.4,-5.852)  circle[radius=1.5pt];
		\node[right = 1mm of {(2.1,-5.64)}] {$q_2$};
		
		%glueing points
		\fill (0.95,-0.9)  circle[radius=1.5pt];
		\fill (1.37,-1.9)  circle[radius=1.5pt];
		%point x
		\fill (1.72,-2.97)  circle[radius=1.5pt];
		\node[right = 1mm of {(1.65,-2.8)}] {$x$};
		%arrows
		\draw [->]  (2.2, -1.4) -- (1.5,-1.8) ;
		\draw [->]  (2.2, -1.4) -- (1.1,-1) ;
		%description
		\node[right = 1mm of {(-0.3,-0.5)}] {$\text{\small 1:1}$};
		\node[right = 1mm of {(0,-1.7)}] {$\text{\small 1:1}$};
		\node[right = 1mm of {(2.2,-1.4)}] {$s$ points in the same fiber as $x$};
		\node[right = 1mm of {(1.7,-0.8)}] {$2:1$};
		\node[right = 1mm of {(2.2,-1.8)}] {$2:1$};
		\end{tikzpicture} \caption{A map [$\pi\colon X \rightarrow \mathbb{P}_1\cup_q \mathbb{P}_2]$ in $\overline{H}_{2i,\mu;3}$ over $[C\cup_{x\sim y}E, \mathcal{O}_C, \eta_E]$ }
	\end{figure} 
If we denote by $s$ the number of points on $C$ in the same fiber as $x$ we get as in \cite{Bud-adm}, Proposition 4.3 the inequalities: 
\[ 2i-1 \leq \mathrm{deg}(\pi_{|C}) + s \leq 2i \]
If the triple ramification is on a rational component, we get the contradiction $\mathrm{deg}(\pi_{|C}) + s  \leq 2i-3$. Hence the triple point is either on $C$ or on $E$. 

\textbf{Case I:} The triple point is on $E$. In this case, the map $\pi_{|E}$ is of degree $2i-2s$ with ramification profiles $(2i-2s-1,1)$, $(2,\ldots,2)$, $(2,\ldots,2)$ and $(3,1\ldots,1)$ over the points $q, q_1,q_2$ and $q_3$. Moreover $\pi_{|C}$ has degree $2i-s-1$, with a point of order $2i-2s-1$ at $x$. 

We know from Proposition \ref{elliptic} that the number of such maps $\pi_{|E}$ which moreover satisfy 
\[ \mathcal{O}_E(\frac{\pi_{|E}^{*}(q_1)- \pi_{|E}^{*}(q_2)}{2}) \cong \eta_E \]
is equal to $4$. As the automorphism of $[E,y]$ is not an automorphism of $\pi_{|E}$, it follows that such admissible covers will appear in our count with multiplicity $2$. Consequently, the contribution coming from \textbf{Case I} is equal to %maybe explain the meaning of a? 
\[ 2\cdot \sum_{s=0}^{i-2}4\cdot a(2i-1-s,2i-2) \]
where $a(d,g) \coloneqq (2d-g-1)\cdot\frac{g!}{d!(g-d+1)!}$. The combinatorial identities in \cite{Bud-adm}, Proposition 2.7 imply: 
\[ 2\cdot \sum_{s=0}^{i-2}4\cdot a(2i-1-s,2i-2) = 8\binom{2i-2}{i}\]

\textbf{Case II:} The triple point is on $C$. In this case, the map $\pi_{|E}$ is of degree $2i-2s$ with ramification profiles $(2i-2s)$, $(2,\ldots,2)$, $(2,\ldots,2)$ and $(2,1,\ldots,1)$. Moreover, the map $\pi_{|C}$ has degree $2i-s$ with order $2i-2s$ at $x$ and another triple point somewhere else. 

The number of such maps $\pi_{|C}$ was computed in \cite{KodevenHarris1984} to be $e(2i-s, 2i-2)$ when $s>0$. We recall that $e(d,g)$ was defined as: 
\[e(d,g) = 8\frac{g!}{(d-3)!(g-d+2)!} - 8\frac{g!}{d!(g-d-1)!} \]
 Proposition 5.4 in \cite{Lian19} implies that this formula is also valid for $s= 0$. The number of such maps $\pi_{|E}$ was computed in \cite{Bud-adm}, Proposition 2.5 to be $6$, and we divide it by $3$ because of the condition   
\[ \mathcal{O}_E(\frac{\pi_{|E}^{*}(q_1)- \pi_{|E}^{*}(q_2)}{2}) \cong \eta_E \]
In this case, $\pi_{|E}$ is fixed by the automorphism of $[E,y]$, hence the multiplicity will be $1$. 

It follows the contribution coming from \textbf{Case II} is 
\[\sum_{s=0}^{i-1}2\cdot e(2i-s,2i-2) = 16(3i-2)\binom{2i-2}{i}\]
In particular, we get 
\[\mathrm{deg}(c_{\mu;3}) = 8\binom{2i-2}{i}+ 16(3i-2)\binom{2i-2}{i} = 24(i-1)\binom{2i-1}{i}\]
\end{proof}

By the theory of double coverings, $\mathcal{R}_g$ can be seen as parametrizing $2:1$ \'etale covers $\pi\colon \tilde{C}\rightarrow C$ where the target curve is smooth of genus $g$. We consider the map $\chi_g\colon \mathcal{R}_g \rightarrow \mathcal{M}_{2g-1}$ sending $[\pi\colon \tilde{C}\rightarrow C]$ to $[\tilde{C}]$ and we are interested in characterizing its image.

 In the case $g=2i$, we know from \cite{FarAprGreencong} that $\mathrm{Im}(\chi_g)$ is contained in the Hurwitz locus $\mathcal{M}^1_{2g-1,g}$ parametrizing curves $[X] \in \mathcal{M}_{2g-1}$ of gonality $\mathrm{gon}(X)\leq g$. Using maps as in Proposition \ref{tripleram}, we get a similar result for the case when $g = 2i-1$. For this, we consider the Hurwitz divisor $\mathfrak{TR}_{2i}$ on $\mathcal{M}_{4i-3}$ parametrizing curves admitting a degree $2i$ map to $\mathbb{P}^1$ with two unspecified triple ramification points, see \cite{FarkasFermat}.

\begin{prop} \label{nontransversality} For $g= 2i-1$ and $i\geq 2$ we have $\mathrm{Im}(\chi_g) \subseteq \mathfrak{TR}_{2i}$.
\end{prop}
\begin{proof}
	Let $[C, \eta] \in \mathcal{R}_g$ and $\pi\colon \tilde{C}\rightarrow C$ its associated double cover. As in Proposition \ref{tripleram}, we consider a degree $2i$ cover $f\colon C \rightarrow \mathbb{P}^1$ having ramification profiles $(2,\ldots,2)$, $(2,\ldots,2)$ and $(3,1,\ldots,1)$ over $0$, $\infty$ and $1$ and satisfying
	\[ \mathcal{O}_C(\frac{f^*(0)-f^*(\infty)}{2})\cong \eta \]  
    We know from \cite{ACGC1}, Appendix B, Exercise 13 that $\tilde{C}$ is isomorphic to the normalization of 
    \[ \left\{(x,y)\in C\times\mathbb{P}^1 \ | \ y^2 = f(x) \right\} \]
    The projection to $\mathbb{P}^1$ given by $(x,y)\mapsto y$ is a degree $2i$ map having two triple ramification points, above $1$ and respectively $-1$.
\end{proof}
In fact, the method of Proposition \ref{nontransversality} can be employed to show that $\mathrm{Im}(\chi_g)$ is contained in many other Hurwitz divisors, for both $g$ odd and $g$ even.

Similarly to Proposition \ref{tripleram}, we prove:
\begin{prop} \label{2prym}
	Let $g=2i-1$ and $[X, z] \in \mathcal{M}_{g,1}$ a generic point. We consider degree $2i$ maps $\pi\colon X\rightarrow \mathbb{P}^1$ such that the ramification profiles over $q_1$ and $q_2$ are $(2,\ldots,2), (2,\ldots,2)$ while $\pi^{-1}(q_3)$ contains $z$ as the unique ramified point of the fiber. The number of such maps which furthermore satisfy
	\[ \mathcal{O}_X(\frac{\pi^{*}(q_1)-\pi^{*}(q_2)}{2}) \cong \eta \]
	for a 2-torsion line bundle $\eta$ of $X$, is equal to $2\binom{2i-1}{i}$. 
\end{prop}
\begin{proof}
	We consider the stack $\overline{\mathcal{CR}}^{\mathrm{ct}}_g \coloneqq\overline{\mathcal{R}}^{\mathrm{ct}}_{g} \times_{\overline{\mathcal{M}}_g} \overline{\mathcal{M}}_{g,1}$ parametrizing pairs $[C,\eta,p]$, where $[C,p]$ is a marked stable curve of compact type and $\eta$ is a line bundle on $C$ that is 2-torsion when restricted to any irreducible component of the curve. 
	
	We define $\overline{H}_{\mu}$ to be the scheme parametrizing admissible covers of degree $2i$ with ramification profile $(2,\ldots,2)$ over two branch points $q_1, q_2$. The order of the simple branch points is irrelevant, except for one which we denote $q_3$. We can naturally define a rational map 
	\[c_{\mu,1}\colon \overline{H}_{\mu} \dasharrow \overline{\mathcal{CR}}^{\mathrm{ct}}_g\]
	sending an admissible cover $[\pi\colon X\rightarrow \mathbb{P}^1]$ to $[X, \mathcal{O}_X(\frac{\pi^{*}(q_1)-\pi^{*}(q_2)}{2}),z]$ where $z$ is the ramified point over $q_3$. Due to the existence of a twist as in \cite{FP18}, this map can be extended over admissible covers with source curve of compact type. 
	
	Our proof reduces again to computing the length of the cycle $c_{\mu,1}^{*}([C\cup_{x\sim y}E, \mathcal{O}_C, \eta_E, z])$ where $[E,y,z] \in \mathcal{M}_{1,2}$ and $[C,x] \in \mathcal{M}_{g-1,1}$ are generic. 
	
	Consider an admissible cover $\pi\colon X\rightarrow \Gamma$ mapped by $c_{\mu,1}$ to this point. The condition imposed on this cover imply that $q_1$ and $q_2$ are contained in $\mathbb{P}_2$, the target of $\pi_{|E}$. As before, we denote by $s$ the number of points on $C$ in the same fiber of $\pi$ as $x$ and we have the inequalities 
	\[2i-1 \leq \mathrm{deg}(\pi_{|C}) + s \leq 2i\] 
	
	If $\mathrm{deg}(\pi_{|C}) =2i-s$, it follows that the order at the node $x$ is $2i-2s$ and hence $\pi_{|E}$ has ramification profiles $(2i-2s), (2,\ldots,2), (2,\ldots,2)$ and $(2,1,1,\ldots,1)$ with the totally ramified point and the ramification point over $q_3$ being generic. This is clearly impossible. 
	
	The only possible case is that $\mathrm{deg}(\pi_{|C}) = 2i-s-1$ and hence $\mathrm{ord}_x(\pi_{|C}) = 2i-2s-1$. In this case, the map $\pi_{|E}$ is as in the hypothesis of Proposition \ref{elliptic2points}.
    
    As a consequence of our description, we get that the degree of $c_{\mu,1}$ is 
	\[ 4\sum_{s=0}^{i-1}a(2i-s-1,2i-2) - 2a(i,2i-2) = 2 \binom{2i-1}{i}\]
\end{proof}

Before we could proceed with our test curve computations, we require one more Prym enumerative result:
\begin{prop} \label{prymunram}
 Let $g=2i-1$, a generic pointed curve $[X,z] \in \mathcal{M}_{g,1}$ and $\eta \in \mathrm{Pic}(X)[2]\setminus \left\{0\right\}$. We consider the maps $\pi\colon X \rightarrow \mathbb{P}^1$ of degree $2i$ such that the ramification profile over two points $q_1$ and $q_2$ is $(2,\ldots,2)$ and $z$ appears as a simple point in the fiber above a branch point $q_3$. The number of such maps satisfying 
 	\[ \mathcal{O}_X(\frac{\pi^{*}(q_1)-\pi^{*}(q_2)}{2}) \cong \eta \]
 is equal to $20(i-1)\binom{2i-1}{i}$.
\end{prop}
\begin{proof}
	Let $[C\cup_{x\sim y} E, \mathcal{O}_C, \eta_E, z] \in \overline{\mathcal{CR}}^{\mathrm{ct}}_g$ as considered in the proof of Proposition \ref{2prym}. Let $\overline{H}_{\mu,1}$ be the Hurwitz stack parametrizing pairs of an admissible cover $\pi\colon X \rightarrow \Gamma$ having ramification profile $(2,\ldots,2)$ over $q_1$ and $q_2$, together with an unramified point $z$ in the fibre over a distinguished simple branch point $q_3$. We consider the rational map 
	\[ c_{\mu,1}\colon \overline{H}_{\mu,1}\rightarrow \overline{\mathcal{R}}_{g,1}\] 
	sending a pair $[\pi\colon X \rightarrow \mathbb{P}^1,z] $ to the stabilization of $[X,z]$ together with the corresponding 2-torsion line bundle induced by the difference between the fibers over $q_1$ and $q_2$. 
	
	Again, we ask what is the length of the fiber over $[C\cup_{x\sim y} E,\mathcal{O}_C, \eta_E, z]$. Let $[\pi\colon X \rightarrow \Gamma, z]$ be a point of the Hurwitz stack in this fibre and let $s$ be the number of points on $C$ in the same fiber of $\pi$ as $x$. It follows that for such an admissible cover the points $q_1, q_2$ are in $\mathbb{P}_2$, the target of $\pi_{|E}$, and $2i-1 \leq \mathrm{deg}(\pi_{|C}) + s \leq 2i$. We distinguish three cases depending on the position of the ramification point above $q_3$. 
    
    \textbf{Case I:} The ramified point over $q_3$ is on a rational component $R$. In this case, the map $\pi_{|E}$ has ramification profiles $(2i-2s-2)$, $(2,\ldots,2)$, $(2,\ldots,2)$ and $(2,1,\ldots,1)$ over $q, q_1,q_2$ and another simple branch point, and the marked point $z$, together with all the points in the same fiber as $z$ are unramified points of $\pi_{|E}$. In this case the map $\pi_{|C}$ is of degree $2i-s-1$ and has ramification order $2i-2s-2$ at $x$ and has a point of ramification order $2$ in the same fibre as $x$. Theorem 2.3 in \cite{Bud-adm}, applied to our case implies the number of choices for the map $\pi_{|C}$ is 
    \[ 2\binom{2i-2}{s-1}(4i-4s-1)\]
    while Proposition 2.5 in \cite{Bud-adm} implies that the number of possible maps $\pi_{|E}$ is $2$. 
    For the special rational component $R$, we have two choices for a degree $4$ map $\pi_{|R}$ having ramification profile $(2,1,1)$ over $q, q_3$ and ramification profile $(2,2)$ over $q_1,q_2$. Both  maps admit an automorphism of order 2. 
    Looking at the complete local ring of such admissible covers in the Hurwitz scheme, we see that when normalizing we get $2$ points above each such cover. In particular, each admissible covers contributes with a value of $2$ to the count and we get a contribution of 
    \[ 8\sum_{s=1}^{i-2}(4i-4s-1)\binom{2i-2}{s-1} = 4(4i-3)\binom{2i-2}{i-1} - 24\binom{2i-2}{i} - 2^{2i} \]
    coming from Case I.
    
    \textbf{Case II:} The ramified point over $q_3$ is on $E$. In this case, the genericity of $[E,y,z]$ implies that there is another ramification point on $E$ that is not mapped to $q, q_1, q_2$ or $q_3$. 
    
    This implies again that $\mathrm{deg}(\pi_{|C}) = 2i-s-1$, the ramification order at $x$ is $2i-2s-1$ and $\pi_{|E}$ has ramification profiles $b_1 = (2i-2s-1,1), b_2 = b_3 = (2,\ldots, 2)$, $b_4=b_5=(2,1,\ldots,1)$ with $0\leq s \leq i-2$.  
    
      Hence, the contribution coming from Case II is 
    \[ \frac{2}{2i-1}\sum_{s=0}^{i-2}(4i-4s-2)(2i-2s-1)\binom{2i-1}{s}\]  
    which is furthermore equal to 
    \[ 2^{2i}-\frac{4}{2i-1}\binom{2i-1}{i}\] 
    
    \textbf{Case III:} The ramified point over $q_3$ is on $C$. In this case, it follows that $\mathrm{deg}(\pi_{|C}) = 2i-s-1$, the order at $x$ is $2i-2s-1$ and the point $z$ is on the rational component collapsing to a point of $E$. The map $\pi_E$ has ramification profiles $b_1 = (2i-2s-1,1), b_2 = b_3 = (2,2,\ldots,2), b_4= b_5 = (2,1,\ldots,1)$ with the points over $q_1$ being both generic. The same method as in the proof of Proposition \ref{elliptic2points} and Proposition \ref{ellipticunramified2pts} can be applied here and we see that such a map can be considered in $12$ ways. Moreover, for every map $\pi_{|C}$ the simple branch point $q_3$ can be chosen in $6i-6$ ways. 
    
    It follows that the contribution in this case is 
    \[ \frac{24(i-1)}{2i-1}\sum_{s=0}^{i-1} (2i-2s-1)\binom{2i-1}{s} = 24(i-1)\binom{2i-2}{i-1}\]
    %which is furthermore equal to 
    %\[ 24(i-1)\binom{2i-2}{i-1}\]
    
    Adding together the three possible cases we get a contribution of 
    \[ 20(i-1)\binom{2i-1}{i}\]
\end{proof} 

In fact, one can alternatively prove Proposition \ref{2prym} and Proposition \ref{prymunram} by considering the fiber over the point $[C\cup_{x \sim y }E, \eta_C, \eta_E, z]$ with $[C,x] \in\mathcal{M}_{g-1,1}$ and $[E,y,z] \in \mathcal{M}_{1,2}$ generic. Similarly for Proposition \ref{tripleram} we can consider the fibre over a generic point $[C/x\sim y, \eta]$ in $\Delta_0''$ to conclude our result. In order to complete this count we need the following result: 

\begin{prop}\label{enumerativeMg} Let $1\leq k \leq i-1$ and $[C,x,y] \in \mathcal{M}_{2i-2,2}$ generic. Then the number of pairs $(L,z)\in G^1_{i+k}(C)\times C$ such that 
\[ h^0(C, L\otimes\mathcal{O}_C(-3z)) \geq 1, \  \mathrm{and} \ h^0(C, L\otimes\mathcal{O}_C(-kx-ky)) \geq 1 \]
with $z \in C\setminus \left\{x,y\right\}$ is equal to $2a(i+k-1,2i-2)+2a(i+k,2i-2)+e(i+k,2i-2)$ if $k \neq 1$ and $2a(i+1,2i-2) +e(i+1,2i-2)$ if $k=1$. Here, $a(d,g)$ and $e(d,g)$ are as before: 
\[a(d,g) \coloneqq (2d-g-1)\cdot\frac{g!}{d!(g-d+1)!} \ \mathrm{and} \ e(d,g) \coloneqq 8\frac{g!}{(d-3)!(g-d+2)!} - 8\frac{g!}{d!(g-d-1)!} \]
\end{prop}	
\begin{proof} The proof follows by degenerating over a point $[C\cup_{t \sim z }\mathbb{P}^1,x,y]$ with $[C,t]\in \mathcal{M}_{2i-2,1}$ generic and $[\mathbb{P}^1,x,y,z] \in \mathcal{M}_{0,3}$.
\end{proof}
Our claim about Proposition \ref{tripleram} follows as we have the identity 
\[4\sum_{s=0}^{i-2}a(2i-s-1,2i-2)+4\sum_{s=0}^{i-1}a(2i-s,2i-2) +2\sum_{s=0}^{i-1}e(2i-s,2i-2) = 24(i-1)\binom{2i-1}{i}\]
\section{Intersections with test curves} \label{sec-intwithtestcurv}
We recall the test curves we considered in \cite{Bud-adm}, appearing also in \cite{Carlos}. Such test curves on $\overline{\mathcal{R}}_g$ can be obtained by pullback of classical ones on $\overline{\mathcal{M}}_g$, which can be found in the literature in \cite{ModHM} and \cite{Mul}. Because $\mathrm{Pic}_\mathbb{Q}(\overline{\mathcal{R}}_g)$ is generated by $\lambda$ and the boundary divisors, see  \cite[Theorem A, Theorem B]{putman} and \cite[Theorem 2.3.1]{MiraBern}, it is enough to describe the intersection of the test curves with this basis.
\subsection{Test curve $A$} We consider the test curve $A$ in $\overline{\mathcal{M}}_g$ consisting of a generic genus $g-1$ curve $C$ glued at a generic point $x$ to a pencil of elliptic curves along a base point. Pulling back the curve $A$ to $\overline{\mathcal{R}}_g$ we obtain three test curves $A_{g-1}, A_1, A_{1:g-1}$, contained in the divisorial components $\Delta_{g-1}, \Delta_1$ and $\Delta_{1:g-1}$ respectively. We have the following intersection numbers with the standard basis of $\mathrm{Pic}_\mathbb{Q}(\overline{\mathcal{R}}_g)$, where the omitted intersections are 0: 
\[ A_{g-1} \cdot \lambda = 1,\  A_{g-1} \cdot \delta_0' = 12, \ A_{g-1} \cdot \delta_{g-1} = -1 \]
\[ A_1 \cdot \lambda = 3, \ A_1 \cdot \delta_0'' = 12, \ A_1 \cdot \delta^{\mathrm{ram}}_0 = 12, \ A_1 \cdot \delta_1 = -3  \]
\[ A_{1:g-1} \cdot \lambda = 3, \ A_{1:g-1} \cdot \delta_0' = 12, \ A_{1: g-1} \cdot \delta^{\mathrm{ram}}_0 = 12, \ A_{1: g-1} \cdot \delta_{1: g-1} = -3 \]

\subsection{Test curve $B$} Let $[C,x] \in \mathcal{M}_{g-1,1}$ generic. The test curve $B$ on $\overline{\mathcal{M}}_g$ is obtained by glueing the point $x$ to a point $y$ moving on the curve. As before, the pullback provides three test curves $B'$, $B''$ and $B^{\mathrm{ram}}$ contained in the divisors $\Delta_0'$, $\Delta''_0$ and $\Delta_0^{\mathrm{ram}}$ respectively. We have the following intersection numbers, the ones omitted being 0: 
\[ B' \cdot \delta_0' = (1-g)(2^{2g}-4), \ B' \cdot \delta_{g-1} = 2^{2g-2}-1, \ B' \cdot \delta_{1: g-1 } =  2^{2g-2}-1 \]
\[ B'' \cdot \delta_1 = 1, \  B'' \cdot \delta_0'' = 2-2g  \] 
\[ B^{\mathrm{ram}} \cdot \delta_0^{\mathrm{ram}} = 2^{2g-2}(1-g), \ B^{\mathrm{ram}} \cdot \delta_1 = 1, \ B^{\mathrm{ram}} \cdot \delta_{1: g-1} = 2^{2g-2}-1 \] 

\subsection{Test curves $C_i$} Let $i$ be an integer satisfying $2\leq i \leq g-1$ and consider two generic curves $[C] \in \mathcal{M}_i$ and $[D,y] \in \mathcal{M}_{g-i,1} $. Let $\eta_C \in \mathrm{Pic}(C)[2]\setminus \left\{0\right\}$ and $\eta_D \in \mathrm{Pic}(D)[2]\setminus \left\{0\right\}$ and consider the test curves in $\overline{\mathcal{R}}_g$ given as
\[[C\cup_{x \sim y } D, (\eta_C, \mathcal{O}_D)]_{x\in C} \]  
\[[C\cup_{x \sim y } D, (\mathcal{O}_C, \eta_D)]_{x\in C} \]  
\[[C\cup_{x \sim y } D, (\eta_C, \eta_D)]_{x\in C} \] 
by varying $x$ along $C$. We denote them $C^i_{i}, C^i_{g-i} $ and $C^i_{i: g-i}$ respectively and it is clear they are  contained in the divisors $\Delta_i, \Delta_{g-i}$ and $\Delta_{i:g-i}$ respectively. The intersection numbers are the following, where all omitted intersection numbers are 0: 
\[ C^i_i \cdot \delta_i = 2-2i \]
\[ C^i_{g-i} \cdot \delta_{g-i} = 2-2i \]
\[ C^i_{i: g-i} \cdot \delta_{i\colon g-i} = 2-2i \] 

\section{The class of the divisor $\overline{D}(\mu;3)$} \label{finalsection}

Using the enumerative results we provided in Section \ref{ellipticsection} and Section \ref{genericcurvesection}, we are now able to compute the intersection of the divisor $\overline{D}(\mu;3)$ with some of the test curves outlined before. 
We will closely follow the treatment appearing in Proposition 4.3, Proposition 4.4 and Proposition 4.5 in \cite{Bud-adm} as the methods we will use are very similar.
\subsection{Intersection with test curves of type $A$} The genericity assumption in the definition of the test curve $A$ immediately imply the following: 
\begin{prop}
	We have the intersection numbers 
	\[\overline{D}(\mu;3)\cdot A_1 =\overline{D}(\mu;3)\cdot A_{g-1}=  \overline{D}(\mu;3)\cdot A_{1:g-1} = 0 \]
\end{prop}

\subsection{Intersection with test curves of type $C_{g-1}$} In this subsection we will compute the intersection of our divisor $\overline{D}(\mu;3)$ with the test curves $C^{g-1}_{1:g-1},C^{g-1}_{g-1}$ and $C^{g-1}_{1}$. 

\begin{prop}
	We have the intersection: 
	\[ \overline{D}(\mu;3)\cdot C^{g-1}_{1:g-1} = (6i-4)!\cdot 48(i-1)(i+1)\binom{2i-1}{i}   \]
\end{prop} 
\begin{proof}
	Consider an admissible cover $\pi\colon X\rightarrow \Gamma$ mapped to a point of $C^{g-1}_{1:g-1}$. The factor $(6i-4)!$ appear as we do not forget the order of the $6i-4$ simple branch points of the admissible cover. For simplicity, we will assume in the computation that these branch points are unordered. 
	We distinguish two cases depending on the position of the point of ramification order 3. 
	
	\textbf{Case I:} The point of ramification order $3$ is on $C$. In this case, the twist and the genericity conditions imply that $q_1$ and $q_2$ are on different rational components of the target. We assume that $q_1$ is on $\mathbb{P}_1$, the target of $\pi_{|C}$. In this case, $\mathrm{deg}(\pi_{|E}) = 2$ and $\pi_{|E}$ is uniquely determined. Moreover, the map $\pi_{|C}$ is a degree $2i$ map with special ramification profiles $(2,\ldots,2), (2,\ldots,2)$ and $(3,1,\ldots,1)$ above three points $q_1, q$ and $q_3$ and the node $x$ is one of the $i$ points in the fibre above $q$. We get in this way a contribution of $2\cdot i\cdot 24(i-1)\binom{2i-1}{i}$. The factor of $2$ appears as we can interchange the roles of $q_1$ and $q_2$.   
	
	\textbf{Case II:} The point of ramification order $3$ is on $E$. Again we assume $q_1$ to be on $\mathbb{P}_1$ and $q_2$ on the other component. In this case, the map $\pi_{|E}$ is of degree $4$ with ramification profiles $(4), (2,2), (3,1)$ and $(2,1,1)$ above $q, q_2, q_3$ and another simple branch point. The map $\pi_{|C}$ is a degree $2i$ map with special ramification profiles $(4,2,\ldots,2)$ and $(2,\ldots,2)$ over $q$ and $q_1$. We showed in Proposition 2.6 in \cite{Bud-adm} that the number of such maps $\pi_{|E}$ is $3$ and clearly $\pi_{|E}$ is not invariant under the automorphism of $[E,y]$. The number of such maps $\pi_{|C}$ satisfying the twist condition
	\[ \eta_C \cong \mathcal{O}_C(\frac{\pi_C^*(q_1)-\pi_C^*(q)}{2})\]
	is immediately observed from \cite{Mul} Section 2.6 to be $4(i-1) \binom{2i-1}{i}$. 
	As a consequence, such admissible covers appear in the count with multiplicity $2$ and we get a contribution number of 
	\[ 2\cdot 2\cdot4\cdot(i-1)\binom{2i-1}{i}\cdot 3 = 48\cdot(i-1)\binom{2i-1}{i} \] 
	Adding the two terms together we get  $48\cdot(i-1)(i+1)\binom{2i-1}{i}$ as required.
\end{proof}

In a completely similar fashion we can compute the intersection of the divisor with $\overline{D}(\mu;3)$ with $C^{g-1}_{g-1}$. 

\begin{prop}
We have the intersection \[ \overline{D}(\mu;3) \cdot C_{g-1}^{g-1} = (6i-4)!\cdot24\cdot(i-1)(6i+2) \binom{2i-1}{i} \]
\end{prop}
\begin{proof}
	We consider an admissible cover $\pi\colon X\rightarrow \Gamma$ mapped to a point of the test curve $C^{g-1}_{g-1}$. Again, we will not take into consideration the order of the simple branch points. In this case, both $q_1$ and $q_2$ are on $\mathbb{P}_1$ and the map $\pi_{|C}$ is again of degree $2i$ and with special ramification profiles $(2,\ldots,2), (2,\ldots,2)$ and $(3,1,\ldots,1)$. If the triple point is on $E$, then the triple point of $\pi_{|C}$ is the node joining $C$ and $E$. In this case we get a contribution of $2\cdot 4\cdot 24(i-1)\binom{2i-1}{i}$. 
	If the triple point is on $C$, it follows that $\pi_{|E}$ has degree $2$ and the node joining $C$ and $E$ is one of the $6i-6$ simple ramification points of $\pi_{|C}$. The contribution in this case is $(6i-6)\cdot 24(i-1)\binom{2i-1}{i}$. 
	Adding the two contributions we get the required sum.
\end{proof}

Lastly, we compute the intersection with the test curve $C^{g-1}_1$. We have the following. 

\begin{prop} We have the intersection
	\[\overline{D}(\mu;3)\cdot C^{g-1}_1 = (6i-4)!\cdot 16\cdot (i-1)(12i^2-i+1)\binom{2i-1}{i}\]
\end{prop}
\begin{proof}
	Consider an admissible cover $\pi\colon X\rightarrow \Gamma$ mapped to a point of the test curve $C^{g-1}_1$. For simplicity, we assume the branch points to be unordered. We distinguish two possible cases, depending on whether the triple point is on $C$ or on $E$. 
	
	\textbf{Case I:} The triple point is on $C$. If we denote by $s$ the number of points on $C$ in the same fibre of $\pi_{|C}$ as $x$ we get that $\pi_{|E}$ has ramification profiles $(2i-2s), (2,\ldots,2), (2,\ldots,2)$ and $(2,1,\ldots,1)$ above $q, q_1, q_2$ and a simple branch point. The map $\pi_{|C}$ is of degree $2i-s$ and has order $2i-2s$ at $x$ and order $3$ somewhere else. The number of possible such maps $\pi_{|E}$ is $2$ and they are invariant under the action of the automorphism of $[E,y]$ and the number of such maps $\pi_{|C}$ is described in Proposition \ref{2i-2s,3}. For simplicity, we will denote the number in Proposition \ref{2i-2s,3} by $f(2i-s,2i-1)$. The contribution coming from this case is equal to 
	\[ \sum_{s=0}^{i-1}2\cdot f(2i-s,2i-1) \]
	In order to compute this term, we use the identities: 
	\[ \sum_{s=0}^{i-1}(2i-2s-1)(2i-2s+1)(i-s-1)\binom{2i-1}{s} = (i-\frac{1}{2})\cdot2^{2i-2}+(2i^2-\frac{5}{2}i)\binom{2i-1}{i}\]
	\[ \sum_{s=0}^{i-1}(2i-2s-1)(2i-2s+1)(i-s+1)\binom{2i-1}{s-1} = (2i^2-\frac{5}{2}i+1)\binom{2i-1}{i} - (i-\frac{1}{2})\cdot2^{2i-2}\]
	\[\sum_{s=0}^{i-1}(2i-2s+2)(s-1)(2i-2s-1)(2i-2s+1)\binom{2i-1}{s-1} = (4i^3-5i^2+3i-2)\binom{2i-1}{i}-(8i^2-12i+4)\cdot2^{2i-2}\]
	\[\sum_{s=0}^{i-1}s(2i-2s)(2i-2s-1)(2i-2s+1)\binom{2i-1}{s} =(4i^3-9i^2+5i)\binom{2i-1}{i}\]
	\[\sum_{s=0}^{i-1}(s+1)(2i-2s-2)(2i-2s-1)(2i-2s+1)\binom{2i-1}{s+1} =(4i^3-5i^2+i)\binom{2i-1}{i}+(8i^2-12i+4)\cdot2^{2i-2}\]
	Adding everything together, the contribution in this case is 
	\[ 16(i-1)(12i^2-3i+1)\binom{2i-1}{i}\]
	
	\textbf{Case II:} The triple point is on $E$. In this case, the ramification profiles of $\pi_{|E}$ are $(2i-2s-1,1)$, $(2,\ldots,2)$, $(2,\ldots,2)$ and $(3,1,\ldots,1)$ over $q, q_1,q_2$ and $q_3$. The map $\pi_{|C}$ is of degree $2i-s-1$ and has ramification order $2i-2s-1$ at the point $x$. The number of such maps $\pi_{|C}$ is known to be $(2i-2s-2)(2i-2s-1)(2i-2s)\binom{2i-1}{s}$, while the number of such maps $\pi_{|E}$ is computed in Proposition \ref{elliptic} to be 12. Such maps $\pi_{|E}$ cannot be invariant under the automorphism of $[E,y]$. Consequently, such admissible covers are counted with multiplicity $2$ and we get a contribution of 
	\[ 32\cdot\sum_{s=0}^{i-1}(i-s-1)(2i-2s-1)(i-s)\binom{2i-1}{s} = 32\cdot i(i-1)\binom{2i-1}{i} \]
	
	Adding together the two cases we get a total contribution of $16\cdot (i-1)(12i^2-i+1)\binom{2i-1}{i}$
\end{proof}
\subsection{Intersection with the test curve $B'$} 

We consider the normalization $\nu\colon \overline{H}_{\mu;3}^\nu \rightarrow \overline{H}_{\mu;3}$. Then, the argument of Proposition 4.1 in \cite{Bud-adm} can be adapted to extend the rational map $c_{\mu;3}\circ \nu\colon \overline{H}_{\mu;3}^\nu \rightarrow \overline{\mathcal{R}}_{2i}$ over the locus parametrizing curves $[C/x\sim y]$ where $[C,x]$ is generic in $\mathcal{M}_{g-1,1}$. Consequently, we know all admissible covers that will be mapped to a point of the test curve $B'$. We have the following 
\begin{prop}\label{B'} We have the intersection 
	\[ \overline{D}(\mu;3) \cdot B' = (6i-4)! \cdot (2^{2g-2}-1)\cdot \left[ 24\cdot(2i-1)(i-1)\binom{2i-1}{i}+72\cdot (i-1)\binom{2i-1}{i} \right]  \]
\end{prop} 
\begin{proof}
	Consider an admissible cover $\pi\colon X\rightarrow \Gamma$ mapping to $B'$. As previously, we will ignore the order of the $6i-4$ simple branch points. For such an admissible cover, we denote by $R$ the rational component collapsing to the node, by $\mathbb{P}_1$ the target of $\pi_{|C}$ and by $\mathbb{P}_2$ the target of $\pi_{|R}$. The genericity of $[C,x]$ imply that $q_1, q_2$ are on $\mathbb{P}_1$, which further implies that $\mathrm{deg}(\pi_{|C}) = 2i$. We distinguish two different cases. 
	
	\textbf{Case I:} The triple ramification point is on $C$. In this case, it follows that $\pi_{|C}$ has special ramification profiles $(2,\ldots,2)$, $(2,\ldots,2)$ and $(3,1,\ldots,1)$ over $q_1, q_2$ and $q_3$ and is unramified at the point $x$. The point $y$ can be chosen in $2i-1$ ways in the same fibre as $x$. The map $\pi_{|R}$ has degree 2 and is unramified at $x$ and $y$. 
	
	We see that $\pi_{|R}$ admits an automorphism fixing $x$ and $y$ but permuting the two ramification points. This would imply that there are just $\frac{(6i-4)!}{2}$ distinct choices for the order of the simple branch points. However, the multiplicity of such a point is $\mathrm{ord}_x(\pi_{|C}) + \mathrm{ord}_y(\pi_{|C}) = 2$. Hence we get in this case a contribution of 
	\[ (2i-1)\cdot(2^{2g-2}-1)\cdot24(i-1)\binom{2i-1}{i}\]
	
	\textbf{Case II:} The triple ramification point is on $R$. In this case we have 
	\[ \mathrm{ord}_x(\pi_{|C}) + \mathrm{ord}_y(\pi_{|C}) = 3\]
	and $\pi_{|R}$ is the unique map with ramification profiles $(2,1), (3)$ and $(2,1)$ over $q, q_3$ and another simple branch point. Depending on whether $x$ is the simple or the ramified point we get the situations of Proposition \ref{2prym} and Proposition \ref{prymunram}. In the case when $x$ is ramified, we can choose a point $y$ in the same fibre of $\pi_{|C}$ in $2i-2$ ways. Moreover, all such admissible covers appear with multiplicity
	\[ \mathrm{ord}_x(\pi_{|C}) + \mathrm{ord}_y(\pi_{|C}) = 3\]
	
	Consequently, the contribution coming from this case is 
	\[ 3\cdot(2i-2)\cdot 2\cdot \binom{2i-1}{i}\cdot (2^{2g-2}-1) + 3\cdot 20\cdot(i-1)\binom{2i-1}{i}\cdot (2^{2g-2}-1) = 72\cdot (i-1)\binom{2i-1}{i}\cdot (2^{2g-2}-1)\]
	
	Adding the cases together we get the conclusion.
\end{proof}

\textbf{Proof of Theorem \ref{divisor}:} From the intersection numbers, we obtain the system of seven equations in seven unknowns that up to the factor of $(6i-4)!$ is: 
\[ a - 4b_0'-4b_0^{\mathrm{ram}}+b_{1:g-1} = a -12b_0'+b_{g-1}= a-4b_0''-4b_0^{\mathrm{ram}} +b_1 = 0\]
\[ b_{1:g-1} = (12i+12), \ \ b_{g-1} = 12\cdot(3i+1), \ \ b_1 =4\cdot (12i^2-i+1) \]	
\[4\cdot (2i-1)b_0' - b_{1:g-1}-b_{g-1} = 24\cdot(2i-1)(i-1)+72\cdot (i-1)\]
Solving the system we get Theorem \ref{divisor}.
\hfill $\square$

We can look at the intersection $\overline{D}(\mu;3)\cdot B''$ and describe the admissible covers mapped to a point of the intersection. We will use here the notations in the proof of Proposition \ref{B'}. 

If $\pi\colon X\rightarrow \Gamma$ is such an admissible cover we have that $q_1, q_2\in \mathbb{P}_2$ and $q_3\in \mathbb{P}_1$. We denote by $2k \coloneqq \mathrm{deg}(\pi_{|R})$. Then $\pi_{|R}$ is the unique map having ramification profiles $(2,\ldots, 2), (2,\ldots,2)$ and $(k,k)$ over $q_1, q_2$ and $q$. Moreover we get that $\pi_{|C}$ has degree $i+k$ and has ramification profiles $(3,1\ldots,1)$ and $(k,k,1,1,\ldots,1)$ over $q_3$ and $q$. We remark that when $k=1$, we also have $i$ choices of the point $y$ in the same fiber as $x$. The description in \cite{Bud-adm}, Section 2.2 of such maps $\pi_{|R}$ implies that each admissible cover appears with multiplicity $2$. 

Using the notations in \cite{FarkasFermat} Proposition 2.1 and Theorem 5.6, we get for $i = 2$ the equality 
\[ 2\cdot 2\cdot b(3,3) + 2\cdot N_3(4) = 6b_0''-b_1 \]  
which is clearly true, thus checking Theorem \ref{divisor} in this case. 

For $i=3$ we derive the following: 
\begin{prop}
	For a general $1$-pointed curve $[C,p] \in \mathcal{M}_{5,1}$ the number of pencils $L \in G^1_6(C)$ satisfying 
	\[ h^0(C, L\otimes\mathcal{O}_C(-3x))\geq 1 \ \mathrm{and} \  h^0(C, L\otimes\mathcal{O}_C(-2p-2y))\geq 1  \]
for some points $x, y \in C\setminus \left\{p\right\}$ is equal to $1560$. 
\end{prop}
\begin{proof}
	We consider the intersection $\overline{D}(\mu;3)\cdot B''$ in genus $6$. On one hand, this is equal to $10b_0''-b_1$ and on the other we get a count for the admissible covers previously described. For $k=1$ we get a contribution of $2\cdot 3 \cdot b(4,5) = 720$. For $k = 2$, the contribution is $2\cdot N_3(5)  = 4080$, see \cite{FarkasFermat} and for $k=3$, the contribution is $2N$ where $N$ is the number of pencils as in the proposition. 
	
	We get the equality:  
	\[720 + 4080 + 2N = 10\cdot 1216 -4240  \]
Hence $N = 1560$ as required.
\end{proof}

 Furthermore, in genus $6$ we have the generically finite Prym map $\mathcal{P}_6\colon \mathcal{R}_6\rightarrow \mathcal{A}_5$, which we can extend to a rational map $P\colon \overline{\mathcal{R}}_6\dashrightarrow \overline{\mathcal{A}}_5$. The pullback and pushforward of the map $P$ have been computed in \cite[Theorem 5]{GruSM-Prym} and \cite[Theorem 7.4]{FarkasGrushevsky}. If we start with the divisor $\overline{D}(\mu;3)$ and we pushforward it to $\overline{\mathcal{A}}_5$, then pullback it back to $\overline{\mathcal{R}}_6$ we get 
\[ P^*P_*[\overline{D}(\mu;3)] = 343440\lambda - 43020\delta_0' - 85860\delta_0^{\mathrm{ram}} \] 
If we subtract the class $[\overline{D}(\mu;3)] = 2112\lambda - 276\delta_0' - 372\delta_0^{\mathrm{ram}}-\cdots$,
we get an equality 
\[ P^*P_*[\overline{D}(\mu;3)]- [\overline{D}(\mu;3)] = 341328\lambda - 42744\delta_0' - 85488\delta_0^{\mathrm{ram}} +\cdots \] 
Because the coefficient of $\delta_0^{\mathrm{ram}}$ is exactly double that of $\delta_0'$ we make the following conjecture. 
\begin{conj} The restriction $\mathcal{P}_{|D(\mu;3)}$ of the Prym map is generically one to one. Moreover, the pullback $\mathcal{P}^{-1}(\mathcal{P}(D(\mu;3)))$ consists of $D(\mu;3)$ and a second divisor that is the pullback of a divisor on $\mathcal{M}_6$ of slope $\frac{2\cdot 547}{137} \approx 7.985$.
	
\end{conj}

\color{black}
\bibliography{main}
\bibliographystyle{alpha}
\Addresses
\end{document}